\documentclass[a4paper,11pt]{amsart}
\usepackage{pdfpages}

\usepackage{amsmath}
\usepackage{amsthm}
\usepackage{amssymb}
\usepackage[latin1]{inputenc}
\usepackage{color}
\usepackage[english]{babel}
\usepackage{graphicx}
\usepackage{subcaption}
\usepackage{tikz}

\theoremstyle{plain}
\newtheorem{thm}{Theorem}[section]
\newtheorem{prop}[thm]{Proposition}

\newtheorem{conj}[thm]{Conjecture}
\newtheorem{cor}[thm]{Corollary}

\theoremstyle{definition}

\newtheorem{ex}[thm]{Example}

\theoremstyle{remark}\newtheorem{rmk}[thm]{Remark}

\newcommand{\CC}{\mathbb{C}}
\newcommand{\ZZ}{\mathbb{Z}}

\newcommand{\br}{\mathbb{R}}

\DeclareMathOperator{\Ric}{Ric}
\DeclareMathOperator{\scal}{scal}

\DeclareMathOperator{\vol}{vol}
\DeclareMathOperator{\area}{A}

\DeclareMathOperator{\sgn}{sgn}

\begin{document}
\title{Minimal hypersurfaces in K\"ahler scalar flat ALE spaces}
\date{\today}

\author[Claudio Arezzo] {Claudio Arezzo}
\address{ICTP Trieste and Univ. of Parma, arezzo@ictp.it }
\author[Alberto Della Vedova] {Alberto Della Vedova}
\address{Universit\`a di Milano-Bicocca, alberto.dellavedova@unimib.it}
\author[Samreena]{Samreena}
\address{samreena01@gmail.com}

\maketitle

\begin{abstract}
In this paper we give a new general method to describe all K\"ahler scalar flat metrics on
$U(n)$-invariant domains of $\CC ^n$ in a way to be able to detect easily whether it can be completed 
to larger domains and which kind of ends they can have. This new approach makes possible to decide whether the
corresponding metrics contain a minimal sphere among the standard euclidean ones. We will show also how to 
check the stability (i.e. minimizing volume up to second order) of such submanifolds.
We apply the present analysis to a comparison between Penrose Inequality (\cite{pen:73}, which requires the presence of a stable minimal 
hypersurface) and Hein-LeBrun's  (\cite{hei:16}, which requires the ambient space to be K\"ahler and predicts the existence of a special divisor), and we will show that Penrose Inequality does indeed hold in any dimension for this class of manifolds, and that the two inequalities are indeed incomparable in that for these manifolds a stable minimal surface and a divisor never coexist.
  
\end{abstract}


\maketitle

\section{Introduction}

Studying the existence and the geometry of minimal submanifolds of a given riemannian manifold is one of the most classical aspects
of Differential Geometry. The main scope of this paper is to make some progress towards the solution of this problem for
K\"ahler scalar flat $U(n)$-invariant metrics. Such metrics appear naturally in a variety of geometric and physical problems especially 
with the further constraint of having an Asymptotically (Locally) Euclidean end (see for example \cite{kr}, \cite{ap},\cite{ap2}, \cite{pen:73}, \cite{hei:16},\cite{vh}). Of particular interest is to decide whether such manifold contain a compact {\em{stable}} (i.e. volume minimizing up to second order) minimal hypersurface. The existence of such objects is of great interest in a number of applications, yet they tend to be very rare, so it is natural to ask whether our curvature conditions prevent their existence. In this respect it is worth noticing that, by Simons Theorem \cite{simons}, the stronger assumption of non-negativity of the Ricci curvature would leave space only to special totally geodesic submanifolds, and at the same time volume minimizers in K\"ahler manifolds are often complex analytic submanifolds (hence never real hypersurfaces), again hinting to a possible negative answer to our problem (see for example \cite{Mic}, \cite{LS}, \cite{AS}, \cite{AM}, \cite{FS} and references therein).

In order to address this problem we have to reformulate the existence theory for scalar flat K\"ahler manifolds in a way suitable to detect 
whether it contains a possibly stable compact real hypersurface among a special class of submanifolds, namely euclidean spheres 
$S^{2n-1}(r)$ in a suitable choice of coordinates on a domain of $\CC ^n$. It is worth noticing that all such submanifolds have constant mean curvature for any $U(n)$-invariant K\"ahler metric (see also Proposition \ref{thm :meancurvature} for a useful explicit expression 
for the mean curvature as a function of $r$), hence forming a global CMC-foliation in the spirit started by Huisken-Yau (\cite{hy}) and later extended and analyzed by many authors (see e.g.  \cite{B}, \cite{Ch}, \cite{CEV}, \cite{H},  \cite{EM},  \cite{EM2}).

The study of $U(n)$-invariant K\"ahler metrics with various curvature properties has a long history, starting with Calabi (\cite{C}) and Tachibana-Liu (\cite{TL}), passing through 
Hwang-Singer (\cite{HS}), Cao (\cite{Cao}), Feldman-Ilmanen-Knopf (\cite{FIK}) and Bryant (\cite{Br}), up to the recent work in the cscK and extremal case by He-Li (\cite{HL}) and Taskent (\cite{T}).

In Section $3$ we present our new approach to the description of scalar flat metrics which extends and recovers by a complete different method previous work in this special case, allowing bounded domains and manifolds with boundary, which are of great interest towards the applications as described below.  This new description turns out to be very well adapted to study the problem of existence of stable minimal $U(n)$-invariant hypersurfaces. Indeed, in Section $4$ we reformulate minimality and stability for a given sphere in this new paradigma.

The output of this study turns out to be surprisingly simple and elegant for the minimal
hypersurface problem 

\begin{thm}
\begin{enumerate}
\item
A complete $U(n)$-invariant scalar flat K\"ahler metric on a manifold without boundary does not contain any compact minimal hypersurface.
\item
There exist manifolds with boundary with complete $U(n)$-invariant scalar flat K\"ahler metrics and an asymptotically euclidean end, which contain a stable minimal hypersphere.
\end{enumerate}
\end{thm}

It would be of great interest to extend the above results removing the $U(n)$-invariance (see e.g. \cite{Trinca} for some partial result in the HyperK\"ahler case). It is important to remark that 
given our examples a possible proof of this non-existence extension cannot follow a local argument as in the above mentioned Simons
Theorem (\cite{simons}) or the many results using the (complex version) of the second variation of volume (for example \cite{Mic}, \cite{LS}, \cite{AS}, \cite{AM}, \cite{FS}) but has to be a much more delicate global argument.
 
Not directly related to the minimal surface problem, it is nevertheless interesting to observe that, arguing as in Subsection 3.1, 
one can prove that 
the only possible ends of a complete scalar flat $U(n)$-metric is $\CC^n/\ZZ^k$ for some $k$, with the only exception of the 
Fu-Yau-Zhu metric (\cite{fu:16}) and its companion described in Section $2$, part $(3)$.

As a general application of our study we compare in Section $5$ two beautiful versions of relations between the ADM mass and the volumes of  special submanifolds. 

The first is the classical one conjectured by Penrose in \cite{pen:73} in the general riemannian setting, and proved by Huisken-Ilmanen (\cite{hui:01}) and Bray \cite{Bra:2001, Bray:2009} up to dimension eight:

\begin{conj} \label{RiemannianPenrose}
Let $(M^d,g)$ be a complete asymptotically euclidean manifold
with non-negative scalar curvature of real dimension $d$, which has an outermost minimal hypersurface $\Sigma$. Then the ADM mass $m_{ADM}$ satisfies the following 
\begin{equation}\label{penrose :inequality}
    m_{ADM} \geq \frac{1}{2}\left(\frac{V(\Sigma)}{V_{E}(\Sigma_{1})}\right)^{\frac{d-2}{d-1}}\,,
\end{equation}
where $V(\Sigma)$ and $V_{E}(\Sigma_{1}) = \frac{2\pi^d}{{(d-1)}!}$ are the volumes of an outermost minimal hypersurface $\Sigma$ and of the euclidean unit hypersphere respectively. Moreover, equality holds if and only if $M^d$ is isometric to a spatial Schwarzschild manifold outside its horizon.
\end{conj}

In the K\"ahler setting Hein-LeBrun (\cite{hei:16}) proved the following

\begin{thm}\label{thm:KPenrose} 
Let $(M,g,J)$ be an asymptotically euclidean complete K\"ahler manifold with non-negative scalar curvature of complex dimension $n$. Then $(M,J)$ carries a canonical divisor $\mathcal{D}$ that is expressed as a sum $\sum n_{i} D_{i}$ of compact complex hypersurfaces with positive integer coefficients together with the property that $\cup D_{i} \ne \emptyset$ whenever $(M,J) \ne \mathbb{C}^{n}$. In term of this divisor, we have 
\begin{equation}\label{eq:of:Penrose:type:}
     m_{ADM} \geq \frac{(n-1)!}{(2n-1) \pi^{n-1} } \sum_{j} n_{j} vol(D_{j}),
\end{equation}
where the equality holds if and only if $(M^{n},g,J)$ is a scalar flat K\"ahler.
\end{thm}

In Section $5$ we will show how our new representation of $U(n)$-invariant scalar flat K\"ahler metrics is very efficient also in computing various riemannian quantities as the ones involved in these Inequalities. By a direct computation we will prove the following  

\begin{cor}\label{phl}
\begin{itemize}

\item
There are no $U(n)$-invariant scalar flat ALE manifolds containing both a stable minimal $S^{2n-1}(t)$ and a divisor.
\item
The Riemannian Penrose Inequality holds in any dimension for all scalar flat K\"ahler metrics containing a stable minimal $S^{2n-1}(t)$.
\end{itemize}

\end{cor}

Again we believe it would be of great interest to generalize the above analysis to more general classes of K\"ahler manifolds to extend this
dicothomy.

\section{Preliminaries on $U(n)$-invariant metrics and their scalar curvature}

In this section we collect few basic well known facts on $U(n)$-invariant K\"ahler metrics for which a large literature is
available (see e.g. \cite{fu:16}, \cite{HL} and \cite{Cao}).
Let $\varphi$ be a K\"ahler potential for a K\"ahler form $\omega = \sqrt{-1}\partial \bar \partial \varphi$ on a  $U(n)$-invariant domain of $\mathbf C^n \setminus\{0\}$.
Suppose there exists a smooth function $f$ on $(a,b) \subset (0,+\infty)$ such that $\varphi(z) = f(|z|^2)$. We have
\begin{equation}\label{omegaspherical} \omega = \sqrt{-1} \left( f''\partial |z|^2 \wedge \bar \partial |z|^2 + f'\partial \bar \partial |z|^2 \right) \end{equation}
For $f(|z|^2)=|z|^2/2$, $\omega_E = \sqrt{-1} \partial  \bar \partial f$ is the Euclidean form.

The volume form of $\omega$ is given by
$$ \frac{\omega^n}{n!} = \frac{(\sqrt{-1})^n}{n!}\left[ n f''(f')^{n-1} \partial |z|^2 \wedge \bar \partial |z|^2 \wedge (\partial \bar \partial |z|^2)^{n-1} + (f')^n (\partial \bar \partial |z|^2)^n \right].$$
Since \begin{equation}\label{|z|^2} \partial \bar \partial |z|^2 = \sum_{i=1}^n dz_i \wedge d \bar z_i, \qquad \partial |z|^2 \wedge \bar \partial |z|^2 = \sum_{i,j=1}^n \bar z_i z_j dz_i \wedge d \bar z_j,\end{equation}
we have
$$ (\partial \bar \partial |z|^2)^n = n! dz_1\wedge d \bar z_1 \wedge \dots \wedge dz_n \wedge d\bar z_n,$$
and 
$$ \partial |z|^2 \wedge \bar \partial |z|^2 \wedge (\partial \bar \partial |z|^2)^{n-1} = (n-1)! |z|^2 dz_1\wedge d \bar z_1 \wedge \dots \wedge dz_n \wedge d\bar z_n,$$
whence
$$ \frac{\omega^n}{n!} = 2^n(f')^{n-1}\left( f'+ |z|^2f''\right) \Omega_E, $$
where $\Omega_E = (\sqrt{-1}/2)^n dz_1\wedge d \bar z_1 \wedge \dots \wedge dz_n \wedge d\bar z_n$ is the Euclidean volume form.

The Ricci form of $\omega$ is given by
\begin{equation*}
 \Ric(\omega) = -\sqrt{-1} \partial \bar \partial \log \frac{\omega^n / n!}{\Omega_E} = -\sqrt{-1}\partial \bar \partial \log \left[ (f')^{n-1}\left( f'+ |z|^2f''\right)\right]. 
 \end{equation*}
 
It will be useful in the sequel to consider the logarithmic change of coordinates $t=\log (|z|^2/2)$, $u(t) = f(|z|^2)$. Indeed, from $f'=\frac{u_t}{|z|^2}=\frac{u_t}{2e^t}$ and $f''= \frac{u_{tt}}{|z|^4} - \frac{u_t}{|z|^4} = \frac{u_{tt}-u_t}{4e^{2t}}$ we get\footnote{Recal that $d= \partial + \bar\partial$, and $d_c=\sqrt{-1}(\bar \partial - \partial)$.}
\begin{eqnarray*}
\omega 
&=& \sqrt{-1} \left( \left(\frac{u_{tt}}{|z|^4} - \frac{u_t}{|z|^4}\right)\partial |z|^2 \wedge \bar \partial |z|^2 + \frac{u_t}{|z|^2}\partial \bar \partial |z|^2 \right) \\
&=& \sqrt{-1} \left( \left(u_{tt} -  u_t\right) \partial \log|z|^2 \wedge \bar \partial \log|z|^2 + u_t \left(\partial \bar \partial \log|z|^2 +\partial \log|z|^2 \wedge \bar \partial \log|z|^2 \right) \right) \\
&=& \sqrt{-1} \left( u_{tt} \partial \log|z|^2 \wedge \bar \partial \log|z|^2 + u_t \partial \bar \partial \log|z|^2 \right)
\end{eqnarray*}
\begin{eqnarray*}
\omega &=& \sqrt{-1} \left( u_{tt} \partial \log|z|^2 \wedge \bar \partial \log|z|^2 + u_t \partial \bar \partial \log|z|^2 \right) \\
&=& \sqrt{-1} \left( u_{tt} \partial t \wedge \bar \partial t + u_t \partial \bar \partial t \right)\\
&=& \frac{1}{2} \left(u_{tt} dt \wedge d_ct + u_t dd_ct \right)
\end{eqnarray*}
and
\begin{equation}\label{volume}
 \omega^n/n! = \frac{u_t^{n-1}u_{tt}}{e^{nt}}\Omega_E
 \end{equation}
 whence
 \begin{equation}\label{Ricci}
\Ric(\omega) = - \sqrt{-1} \partial \bar \partial \log \left( \frac{u_t^{n-1}u_{tt}}{e^{nt}} \right).
 \end{equation}

\begin{ex}\label{ex::ricci_flat}
Let $n=2$, and for all $b\geq 0$ let 
$$u(t) = b\log\frac{\sqrt{e^{2t}+b^2}-b}{\sqrt{e^{2t}+b^2}+b}+2\sqrt{e^{2t}+b^2}$$
Note that
$$ u(t) = 2bt -2 b\log(\sqrt{e^{2t}+b^2}+b)+2\sqrt{e^{2t}+b^2}.$$
Thus by \eqref{Ricci} the K\"ahler form
$$ \omega_b = \sqrt{-1} \partial \bar \partial \left( b\log\frac{\sqrt{|z|^4+4b^2}-2b}{\sqrt{|z|^4+4b^2}+2b}+\sqrt{|z|^4+4b^2}\right)$$
is Ricci-flat on $\mathbf C^2 \setminus\{0\}$ for all $b \geq 0$.
For $b=0$ it reduces to twice the Euclidean metric, while for $b=\frac{1}{2}$
we recognize the Eguchi-Hanson metric (see \cite[Example 7.2.2]{Joyce}).
\end{ex}

By Formula \eqref{Ricci}, one has
$$ \scal (\omega) = -\Delta \log \left[ \frac{u_t^{n-1}u_{tt}}{|e^{nt}} \right], $$
whence

$$ \scal (\omega) = (n-1) n\frac{1}{u_t} - (n-1)(n-2) \frac{u_{tt}}{u_t^2} - 2(n-1) \frac{u_{ttt}}{u_tu_{tt}} + \frac{u_{ttt}^2}{u_{tt}^3} - \frac{u_{tttt}}{u_{tt}^2}. $$
Since positivity of $\omega$ forces $u_t$ and $u_{tt}$ to be positive, one can set $v=\log u_t$ (so that $v_t > 0$) and get

\begin{eqnarray} 
v_te^v\scal (\omega)
&=& (n-1) nv_t - (n-1)(n-2) v_t^2 - 2(n-1) \left(v_{tt}+v_t^2\right) \\
&& + \frac{(v_{tt}+v_t^2)^2}{v_t^2} - \frac{v_{ttt}+3v_tv_{tt}+v_t^3}{v_t}
\end{eqnarray}

or equivalently
\begin{eqnarray} 
\label{scalar}
v_te^v \scal = (n-1) n (1-v_t)v_t - (2n-1) v_{tt} - \left(\frac{v_{tt}}{v_t}\right)_t.
\end{eqnarray}

\begin{ex}
Let $n=2$ and, for all $b \geq 0$ consider
\begin{equation*}
u(t) = bt + e^t \,\, .
\end{equation*}
Then
\begin{equation*}
\omega_b = \sqrt{-1}\partial \bar \partial \left( \frac{|z|^2}{2}  + b\log|z|^2 \right)
\end{equation*}
and thus 
\begin{equation*}
\scal (\omega)
= \frac{2be^t}{(b+e^t)^2} - \frac{3be^t}{(b+e^t)^2} + \frac{be^t}{(b+e^t)^2} = 0 \,\, .
\end{equation*}
In fact $\omega_b$ extends to the Burns-Simanca K\"ahler metric \cite{sim:91} on the blow-up at origin of $\mathbb C^2$.
\end{ex}

\begin{ex}
An important example of a {\em{complete}} scalar flat K\"ahler metric on $\mathbb C^n \setminus\{0\}$ with two ends has been recently discovered by 
Fu-Yau-Zhou in \cite{fu:16}.  Their result is the following

\begin{thm} \label{FYZ}
There exist a one parameter family of functions $t\rightarrow{f_{a}(t)}$ defined on $\br$ and smoothly depending on the parameter $a>0$, such that the metric associated to the K\"ahler form $\omega_{f}$ is complete and scalar flat on $\CC^{n}\setminus{0}.$ Moreover, the function $f_{a}$ has the following expansion as $t\rightarrow{\infty}$
\begin{equation*}
 f_{a}(t)=
    \begin{cases}
    |w|^2-\frac{na^{n-1}}{(n-1)(n-2)}|w|^{4-2n}+\frac{a^n}{n}|w|^{2-2n}+O(|w|^{-2n}) \quad \text{for} \quad n \geq 3\\
    |w|^2+2 a \log|w|^2+\frac{a^2}{2|w|^2}+O(|w|^{-4}) \quad \text{for} \quad n=2,.
\end{cases}
\end{equation*}
In particular, the metric is AE at infinity, and as $t \to -\infty$, we have the following expansion
\begin{equation*}
    f_{a}(t)=a\log|w|^2 -\frac{2a}{n(n-1)}\log(-\log|w|^2)+O(\frac{1}{\log|w|^2})
\end{equation*}
where $t=\log|w|^2.$
\end{thm}

This metric will play a significant r\^ole in the subsequent sections, and we will refer to it as the {\em Fu-Yau-Zhou metric}.
\end{ex}

\section{Scalar-flat $U(n)$-invariant metrics}
In the case of scalar flat metrics, equation (\ref{scalar}) above reduces to a second order ODE on the function $v_t$. Therefore, once we set
\begin{equation}\label{eq::system1}
 \left\{
\begin{array}{ccc}
x = (1-n)v_t-\frac{v_{tt}}{v_t}+n-1\\
y=nv_t+\frac{v_{tt}}{v_t}-n,
\end{array}
\right.
\end{equation}
one has $x+y+1=v_t$ and  is reduced to solve the first order system 
\begin{equation}\label{eq::system}
\left\{
\begin{array}{ccc}
x_t =&-nx(x+y+1) \\ 
y_t =&(1-n)y(x+y+1). 
\end{array}
\right.
\end{equation}
subject to the constraint $x+y+1>0$.
Clearly this is equivalent to find integral curves of the vector field 
\begin{equation*}
V =  -n(1+x+y)x \frac{\partial}{\partial x} + (1-n)(1+x+y)y \frac{\partial}{\partial y}
\end{equation*}
in the half plane above the line $x+y+1-0$.
Once $(x,y)$ is the integral curve of $V$ with initial point $(x_0,y_0)$, defiined on the open $(a,b)\subset \mathbf R$ for some  $a\in [-\infty,0)$, $b \in (0,+\infty]$,
 by
\begin{equation} \label{vinteg}
v(t) = \int_0^t \left(1+x(\tau)+y(\tau)\right)d\tau
\end{equation}
is defined $u_t=e^v$, that is a scalar flat K\"ahler metric on the domain $\{e^a<|z|^2/2<e^b\}\subset \mathbf C^n$.

Note that $V$ is singular at all points of the line $x+y+1=0$. Moreover, $(0,0)$ is an isolated singular point for $V$, and any other point is regular.
Finally, note that the axes $x=0$, and $y=0$ are stable separatices at the singular point $(0,0)$.
We are interested in integral curves of $V$ which lie above the line $x+y=-1$, since each of them corresponds to a scalar flat $U(n)$-symmetric K\"ahler metric on domains in $\mathbf C^n$.

\begin{ex}\label{ex::E}
The metric related to the singular point $(x_0,y_0)=(0,0)$ is related to the Euclidean one. 
This follows by the equation $v_t=1$ which gives $u_t = e^t$, whence $u = e^t$.
Thus $u$ is the potential of the Euclidean metric on $\mathbf C^n$.
\end{ex}

\begin{ex}\label{ex::BS}
Consider the solution of system \eqref{eq::system}  with initial point $(0,y_0)$, $y_0>-1$.
Clearly the first equation yield solution $x = 0$. 
On the other hand, the second equation yield
$ y_t = (1-n)y(1+y) $ whence
\begin{equation*}
\left( \frac{1}{y} - \frac{1}{1+y} \right) dy = (1-n) dt
\end{equation*}
and finally
\begin{equation*}
\log|y| - log|1+y| - \log|y_0| + log|1+y_0| = (1-n)t,
\end{equation*}
\begin{equation*}
\frac{y}{1+y} = \frac{y_0}{1+y_0} e^{(1-n)t},
\end{equation*}
whence, leting $B=-y_0/(1+y_0)$, so that $B>-1$, one has
\begin{equation*}
y = \frac{-Be^{(1-n)t} }{ 1+ Be^{(1-n)t} }
\end{equation*}
Moreover one can calculate the integral curve 
\begin{equation*}
(x,y) = \left(0, \frac{-B }{ B+e^{(n-1)t} }\right)
\end{equation*}
For $B\geq 0$ (i.e. $-1< y_0 \leq 0$)  it exists for all $t \in \mathbf R$.
For $B=y_0=0$ it is the constant curve at point $(0,0)$ corresponding to the Euclidean metric.
For $-1<B<0$ (i.e. $y_0> 0$) the curve exists for $t>log\sqrt[n-1]{-B}$.
It follows that
\begin{equation*}
v
= \int_0^t (1+\frac{-B}{B+e^{(n-1)\tau}})d\tau
= \frac{1}{n-1}\log(B+e^{(n-1)t})
\end{equation*}
which implies that 
\begin{equation*}
u_t = \sqrt[n-1]{B+e^{(n-1)t}},
\end{equation*}
is the derivative of the potential of a scalar flat curvature metric on the domain $|z|^2>2\sqrt[n-1]{-B}$, $\mathbb C^n$, $\mathbb C^n \setminus\{0\}$ according to $-1<B<0$, $B=0$, or $B>0$. 
For $n=2$ and $B>0$ this gives, for an unimportant constant $C\in\mathbb R$, the potential of the Burns-Simanca metric
\begin{equation*}
u= Bt+e^t +C.
\end{equation*}
\end{ex}

\begin{ex}\label{ex::HH}
Consider the solution of system \eqref{eq::system}  with initial point $(x_0,0)$, $x_0>-1$.
Clearly the second equation yield solution $y = 0$. 
On the other hand, the first equation yield
$ x_t = -n(1+x) $ whence
\begin{equation*}
\left( \frac{1}{x} - \frac{1}{1+x} \right) dx = -n dt
\end{equation*}

and finally
\begin{equation*}
\log(x) - log|1+x |- \log(x_0) + log|1+x_0| = -nt,
\end{equation*}

whence, leting $B=-x_0/(1+x_0)$, so that $B>-1$, one has
\begin{equation*}
x = \frac{-Be^{-nt} }{ 1+ Be^{-nt} }
\end{equation*}
Moreover one can calculate the integral curve 
\begin{equation*}
(x,y) = \left(\frac{-B }{ B+e^{nt} },0\right)
\end{equation*}

For $B\geq 0$ (i.e. $-1< x_0 \leq 0$)  it exists for all $t \in \mathbf R$.
For $B=x_0=0$ it is the constant curve at point $(0,0)$ relating to the Euclidean metric.
For $-1<B<0$ (i.e. $x_0> 0$) the curve exists for $t>log\sqrt[n]{-B}$.

It follows that
\begin{equation*}
v
= \int_0^t (1+\frac{-B}{B+e^{n\tau}})d\tau
= \frac{1}{n}\log(B+e^{nt})
\end{equation*}
which implies that
\begin{equation*}
u_t =\sqrt[n]{B+e^{nt}},
\end{equation*}
is the derivative of the potential of a scalar flat curvature metric on the domain $|z|^2>\sqrt[n]{-B}$, $\mathbb C^n$, $\mathbb C^n \setminus\{0\}$ according to $-1<B<0$, $B=0$, or $B>0$. 
Thanks to Example \ref {ex::ricci_flat}, for $n=2$ this gives the potential of a Ricci flat metric. Specifically, it gives half the potential of the Eguchi-Hanson metric when $B=1/4$,  
\begin{equation*}
u= \frac{t}{2} -\frac{1}{4} \log(\sqrt{1+4e^{2t}}+1)+\frac{1}{2}\sqrt{1+4e^{2t}}.
\end{equation*}
\end{ex}

The above examples can be summarized in the following
\begin{prop} \label{solutionaxis}
Let $n \geq2$ and let $k \in \{n-1,n\}$. For any constant $A>0$, $B>-1$ let
\begin{equation*}
u_t(t) = A \sqrt[k]{B+e^{kt}}.
\end{equation*}
Letting $\omega = \sqrt{-1} \partial \bar \partial u(t)$ defines a K\"ahler metric on the domain $|z|^2 > \sqrt[k]{-B}$, $\mathbf C^n$, or $|z|^2>0$ according to $-1<B<0$, $B=0$, or $B>0$ respectively. Moreover, $\omega$ is scalar-flat for all $k,A,B$, it is Ricci-flat whenever $k=n$, and it is flat if $B=0$.  
\end{prop}

By trivial calculations one finds that the metrics given in the Proposition above are explicitly given by
\begin{equation*}
u_t = A \sqrt[k]{B+e^{kt}},
\qquad u_{tt} = A \frac{e^{kt}\sqrt[k]{B+e^{kt}}}{B+e^{kt}} 
\end{equation*}
\begin{align*}
\omega 
&= A \sqrt[k]{B+e^{kt}} \left( \sqrt{-1} \partial \bar \partial t
+ \frac{e^{kt}}{B+e^{kt}} \sqrt{-1} \partial t \wedge \bar \partial t \right) \\
& = A \sqrt[k]{1 + B |z|^{-2k}} \left( \sqrt{-1} \partial \bar \partial |z|^2
- \frac{B|z|^{-2k-2}}{1 + B |z|^{-2k}} \sqrt{-1} \partial |z|^2 \wedge \bar \partial |z|^2 \right),
\end{align*}

wherever they are defined.
On the other hand, since $u_t$ is real analytic, for $|z|^2 > \sqrt[k]{|B|}$ any K\"ahler potential of $\omega$ has the expansion
\begin{align*}
u(t) =& C + A \sum_{q \geq 0} \binom{1/k}{q} B^q \frac{e^{(1-qk)t}}{1-qk} \\
=& C + A \left\{
e^t
- \frac{B}{k(k-1)} e^{(1-k)t}
+ \frac{B^2(k-1)}{k^2(2k-1)} e^{(1-2k)t}
+ O(e^{(1-3k)t})
\right\},
\end{align*}
for some un-important constant $C$.
This proves that $\omega$ is, up to a scaling factor, an ALE K\"ahler metric for all $k,A,B$.

In order to investigate geometric properties of $\omega$ near the boundary of its domain, we distinguish two cases, according to the sign of $B$ (leaving apart the trivial case $B=0$, corresponding to the euclidean metric, up to a scaling factor). 

\subsection{Metrics with $B>0$} \label{inco}
While $\omega$ is incomplete on $\mathbf C^n \setminus \{0\}$, it induces a (rescaled) ALE K\"ahler metric on the quotient by the cyclic group $\Gamma_k=\mathbf Z/k \mathbf Z$ of the blow-up of $\mathbf C^n$ at the origin.
To see this, let $Bl_0 \mathbf C^n = \{(z,w) \in \mathbf C^n \times \mathbf {CP}^{n-1} \,|\, z_iw_j-z_jw_i=0 \mbox{ for all }1 \leq i<j \leq n\}$ be the blow-up of $\mathbf C^n$ at the origin.
As well known, to any affine open set $w_i \neq 0$ of $\mathbf {CP}^{n-1}$ corresponds an open set of $Bl_0 \mathbf C^n$, which is biholomorphic to $\mathbf C^n$ via the map 
$$(\lambda,\zeta) \mapsto \left((\lambda\zeta_1, \dots, \lambda\zeta_{i-1},\lambda, \lambda\zeta_i, \dots ,\lambda\zeta_{n-1}), (\zeta_1, \dots, \zeta_{i-1},1, \zeta_i, \dots ,\zeta_{n-1})\right)$$ 
(note that here $\lambda \in \mathbf C$ and $\zeta \in \mathbf C^{n-1}$).
In these coordinates one has
\begin{multline*}
\omega = A \sqrt[k]{ B + |\lambda|^{2k}(1+|\zeta|^2)^k } \left( \sqrt{-1} \partial \bar \partial \log(1+|\zeta|^2) \vphantom{ \frac{|\lambda|^{2k}(1+|\zeta|^2)^k}{B+|\lambda|^{2k}(1+|\zeta|^2)^k}} \right.\\
\left. + \frac{|\lambda|^{2k}(1+|\zeta|^2)^k}{B+|\lambda|^{2k}(1+|\zeta|^2)^k} \sqrt{-1} \partial \log\left(|\lambda|^2(1+|\zeta|^2)\right) \wedge \bar \partial \log\left(|\lambda|^2(1+|\zeta|^2)\right) \right).
\end{multline*}
Since $k$ is a positive integer and 
\begin{multline*}
|\lambda|^{2k} \sqrt{-1} \partial \log\left(|\lambda|^2(1+|\zeta|^2)\right) \wedge \bar \partial \log\left(|\lambda|^2(1+|\zeta|^2)\right) = \\ |\lambda|^{2k-2} \sqrt{-1} \left( d\lambda + \lambda \partial \log(1+|\zeta|^2)\right) \wedge \left( d\bar\lambda + \bar\lambda \bar \partial \log(1+|\zeta|^2)\right),
\end{multline*}
$\omega$ is smooth, but it is degenerate at points $\lambda=0$ in the direction of $\lambda$ as soon as $k \geq 2$.
This is the reason why one has to take the quotient by the cyclic group $\Gamma_k$ in order to get a K\"ahler metric on $M^n_k = \Gamma_k \backslash Bl_0 \mathbf C^n$.
Note that the action of $\Gamma_k$ on $Bl_0 \mathbf C^n$ is defined by $\gamma \cdot (z,w) = (e^\frac{2\pi i \gamma}{k}z,w)$.
The quotient map from $Bl_0 \mathbf C^n$ to $M^n_k$ maps $(\lambda,\zeta)$ to $(\lambda^k,\zeta)$.
Therefore, in local coordinates $(\mu,\zeta)$ on $M^n_k$ one has
\begin{multline*}
\omega' = A \sqrt[k]{ B + |\mu|^2(1+|\zeta|^2)^k } \left( \sqrt{-1} \partial \bar \partial \log(1+|\zeta|^2) \vphantom{ \frac{|\lambda|^{2k}(1+|\zeta|^2)^k}{B+|\lambda|^{2k}(1+|\zeta|^2)^k}} \right.\\
\left. + \frac{k^{-2} |\mu|^2(1+|\zeta|^2)^k}{B+|\mu|^2(1+|\zeta|^2)^k} \sqrt{-1} \partial \log\left(|\mu|^2(1+|\zeta|^2)^k\right) \wedge \bar \partial \log\left(|\mu|^2(1+|\zeta|^2)^k\right) \right),
\end{multline*}
whence it follows that $\omega'$ is a smooth scalar-flat K\"ahler metric on $M^n_k$ for $k \in \{n-1,n\}$ and it is Ricci-flat if $k=n$.
For $n=2$, it is, up to scalar factors, the Burns-Simanca metric on $Bl_0\mathbf C^2$ when $k=1$ and the Eguchi-Hanson metric on $\Gamma_2 \backslash Bl_0\mathbf C^2$ when $k=2$.

\subsection{Metrics with $-1<B<0$}
Even in this case, $\omega$ is an incomplete K\"ahler metric on the domain $|z|^2 > \sqrt[k]{-B}$, and it degenerates when $|z|^2$ approaches $\sqrt[k]{-B}$.
To see this, note that $\omega$ can be written in the form
\begin{equation*}
\omega = A \sqrt[k]{1 + B |z|^{-2k}} \left( \sqrt{-1} \partial \bar \partial |z|^2
- \frac{4B}{|z|^{2k} + B} \sqrt{-1} \partial |z| \wedge \bar \partial |z| \right).
\end{equation*}
Moreover, recall that $\sqrt{-1} \partial \bar \partial |z|^2$ is the K\"ahler form of twice the Euclidean metric $g_E = dr^2 + r^2g_{S^{2n-1}}$, being $r=|z|$, and that $\sqrt{-1} \partial |z| \wedge \bar \partial |z| = \frac{1}{2} dr \wedge d^cr$, being $d^c=-dr \circ J$, and $J$ the standard complex structure on $\mathbf C^n$.
As a consequence, the metric associated with $\omega$ is given by
\begin{align*}
g 
=& 2A \left( 1 + B r^{-2k} \right)^{1/k} \left( dr^2 + r^2g_{S^{2n-1}}
- \frac{B}{r^{2k} + B} (dr^2+d^cr^2) \right) \\
=& 2Ar^{-2} \sqrt[k]{ r^{2k} + B } \left( dr^2 + r^2g_{S^{2n-1}}
- \frac{B}{r^{2k} + B} (dr^2+d^cr^2) \right) \\
=& 2Ar^{-2} \sqrt[k]{ r^{2k} + B } \left( dr^2 + r^2g_{S^{2n-1}} \right)
- \frac{2ABr^{-2}}{\left(r^{2k} + B\right)^\frac{k-1}{k}} (dr^2+d^cr^2) \\
=& \frac{2A r^{-2}}{\left( r^{2k} + B \right)^\frac{k-1}{k}} \left( r^{2k}dr^2 + (r^{2k} + B)r^2g_{S^{2n-1}}
- B d^cr^2 \right)
\end{align*}
\begin{equation*}
g = \frac{2A r^{-2}}{\left( r^{2k} + B \right)^\frac{k-1}{k}} \left( r^{2k}dr^2 + (r^{2k} + B)r^2g_{S^{2n-1}}
- B d^cr^2 \right).
\end{equation*}
Consider now a path $\rho: [0,1) \to \mathbf C^n$ defined by $\rho(t) = (2-t)a$, for some fixed $a$ belonging to the sphere of radius $\sqrt[2k]{-B}$ in $\mathbf C^n$.
As one can readily check, the norm of $\dot \rho(t)$ is asymptotic to $\sqrt{2A} (-B)^\frac{3}{2k}(2k)^\frac{1-k}{2k}(1-t)^\frac{1-k}{2k}$ as $t \to 1$, whence it follows that the length of $\rho$ is finite.
As anticipated, this proves that the metric $g$ is incomplete.
Now consider the circle $\gamma: [0,2\pi] \to \mathbf C^n$ defined by $\gamma(t) = e^{\sqrt{-1}t}a$, where $\ell>0$ and $a\in \mathbf C^n$ satisfies $|a|^2 > \sqrt[k]{-B}$.
Due to $U(n)$-invariance of $g$, the norm of $\dot \gamma$ turns out to be constant, and it is asymptotic to $\sqrt{2A} (-B)^\frac{3}{2k}k^\frac{1-k}{2k} (|a|^2-\sqrt[k]{-B})^\frac{1-k}{2k}$ as $|a|^2 \to \sqrt[k]{-B}$.
As a consequence, whenever $k \neq 1$, the boundary of the domain $|z|^2 < \sqrt[k]{-B}$ collapses to a circle of infinite length.

\subsection{Level sets}

Solutions of the system \eqref{eq::system} lie in the level curves of the function 
\begin{equation}
F(x,y)= |x|^{1-n}|y|^n.
\end{equation}
In fact, by equations on the system one has
$$ \frac{dy}{dx} = \frac{(n-1)y}{nx}
\qquad \frac{ndy}{y} = \frac{(n-1)dx}{x}$$

$$ 
\quad n(\log|y|-\log|y_0|) = (n-1)(\log|x|-\log|x_0|) $$
\begin{equation}
 |y|=|y_0||x_0|^{1/n-1}|x|^{1-1/n}
 \end{equation} 
 
 In other words, letting $\lambda=|x_0|^{1-n}|y_0|^n \in [0,\infty]$, the level curves of $F$ are the coordinate lines $x=0$ ($\lambda=\infty$), $y=0$ ($\lambda=0$), and graph of the functions $\phi(x)=\sgn(y_0) \lambda^{1/n}|x|^{1-1/n}$ with $\lambda \neq \infty$, $\psi(y) = \sgn(x_0) \lambda^{1/(1-n)} |y|^{n/(n-1)}$ with $\lambda \neq 0$.
Clearly level curves of $F$ have symmetry with respect to the two coordinate lines $xy=0$.
The solutions of \eqref{eq::system} are implicitely given by

\begin{equation} \label{timespanx}
\int_{x_0}^{x(t)} \frac{dx}{-nx(1+x+\sgn(y_0) \lambda^{1/n}|x|^{1-1/n})}  =  t
\end{equation}

or equivalently 

\begin{equation} \label{timespany}
\int_{y_0}^{y(t)} \frac{dy}{-(n-1)y(1+y+\sgn(x_0) \lambda^{1/(1-n)}|y|^{n/(n-1)})}  =  t\, .
\end{equation}

The scalar flat metrics relative to solutions of the system \eqref{eq::system} behave according to initial point $(x_0,y_0) \ \in \mathbf R^2$  such that $x_0+y_0+1>0$ and consequently by $\lambda \in [0,\infty]$. 
Let $\lambda_{FYZ}= n^n/(n-1)^{n-1}$. Such value is of critical importance in our analysis as it it the value for which the level set 
$\{F=\lambda\}$ is tangent to the line $\{1+x+y =0\}$ which is the discriminant to decide whether (a piece of) a level set does correspond to a riemannian metric, and will be hence called {\em line of admissible metrics}.

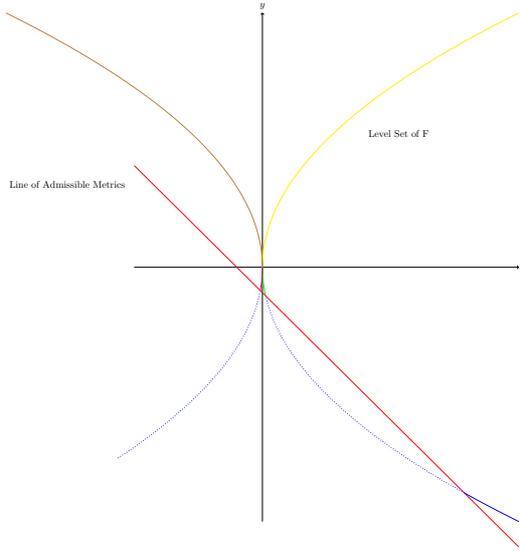
\begin{figure}
\centering
\resizebox{7cm}{!}{
\begin{tikzpicture}
  \draw[->] (-5,0) -- (10,0) node[right] {$x$};
  \draw[->] (0,-10) -- (0,10) node[above] {$y$};

  \draw[red,thick] (-5,4) -- (10,-11);
    \node at (-10, 3) [above right] {Line of Admissible Metrics};

  \draw[blue,thick,domain=-10:-8.87,smooth,variable=\x] plot ({0.1*\x*\x}, {\x});
  \draw[blue,thick,domain=-8.87:-1.13,smooth,variable=\x,dotted] plot ({0.1*\x*\x}, {\x});
  \draw[green,thick,domain=-1.1:0,smooth,variable=\x] plot ({0.1*\x*\x}, {\x});
   \draw[magenta,thick,domain=-0.93:0,smooth,variable=\x] plot ({-0.1*\x*\x}, {\x});
    \draw[blue,thick,domain=-7.5:-1,smooth,variable=\x,dotted] plot ({-0.1*\x*\x}, {\x});
   \draw[yellow,thick,domain=0:10,smooth,variable=\x] plot ({0.1*\x*\x}, {\x});
    \draw[brown,thick,domain=-10:0,smooth,variable=\x] plot ({-0.1*\x*\x}, {-\x});
  \node at (4, 5) [above right] {Level Set of F};

\end{tikzpicture}
}
\caption{Level set of $F$ divided into the pieces corresponding to different scalar flat metrics. The dotted part below the line of admissible metrics does not define a metric. This level set defines (up to scaling) $5$ different scalar flat metrics, all ALE except for the blue one.} \label{fig:M1}
\end{figure}

\begin{enumerate}
\bigskip

\item $x_0y_0=0$ 

These cases are covered by the examples above.
($x_0=0$ like Burns-Simanca, $y_0=0$ like Eguchi-Hanson)

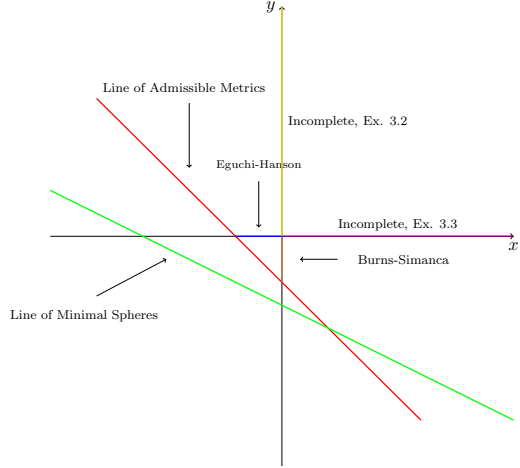
\begin{figure}
\centering
\resizebox{7cm}{!}{
\begin{tikzpicture}
  \draw[->] (-5,0) -- (5,0) node[below] {$x$};
  \draw[->] (0,-5) -- (0,5) node[left] {$y$};

  \draw[red,thick] (3,-4) -- (-4,3);
     \node at (-4, 3) [above right] {{\scriptsize Line of Admissible Metrics}};

  \draw[green,thick] (-5,1) -- (5,-4);
    \node at (-6, -2) [above right] {{\scriptsize Line of Minimal Spheres}};
    
      \draw[brown,thick] (0,0) -- (0,-1);
    \node at (1.5,-0.5) [right] {{\scriptsize Burns-Simanca}};
    
        \draw[blue,thick] (0,0) -- (-1,0);
    \node at (-0.5,1.3) [above] {{\tiny Eguchi-Hanson}};
    
          \draw[yellow,thick] (0,0) -- (0,5);
    \node at (0,2.5) [right] {{\scriptsize Incomplete, Ex. 3.2}};
    
         \draw[violet,thick] (0,0) -- (5,0);
    \node at (2.5,0) [above] {{\scriptsize Incomplete, Ex. 3.3}};

    \draw[->][very thin] (-0.5,1.2) -- (-0.5,0.2);
\draw[->][very thin] (1.2,-0.5) -- (0.4,-0.5);
\draw[->][very thin] (-2,2.9) -- (-2,1.5);
\draw[->][very thin] (-4,-1.3) -- (-2.5,-0.5);

\end{tikzpicture}
}
\caption{Positioning in the $(x,y)$-plane of the Examples of Proposition \ref{solutionaxis}} \label{fig:M2}
\end{figure}

\bigskip

\item\label{item::2} $x_0>0$, $\lambda< \lambda_{MY}$, or $x_0,y_0>0$

The curve of the vector field $V$ run along the graph of $\phi$. Since $V$ is asymptotic to $V_0 = -nx\frac{\partial}{\partial x} +(1-n)y\frac{\partial}{\partial x}$ around the stable note $(0,0)$. therefore $\lim_{t \to +\infty} (x,y) = (0,0)$. More specifically, 
\begin{equation*}
(x,y) \sim (x_0e^{-nt},y_0e^{(1-n)t}) \quad \mbox{as} \quad t \to +\infty.
\end{equation*} 
 
Since $v_t =1 + x + y$, by \eqref{eq::system} one has
\begin{align*}
v = \int_0^t \left(1+x(\tau)+y(\tau)\right) d\tau & = \int_{x_0}^{x(t)} \frac{dx}{-nx} = \int_{y_0}^{y(t)} \frac{dy}{(1-n)y} \\
& =-\frac{1}{n}\log|x| + \frac{1}{n} \log|x_0| = -\frac{1}{n-1}\log|y| + \frac{1}{n-1} \log|y_0|
\end{align*}
whence $ u_t = \sqrt[n]{\left|\frac{x_0}{x}\right|} = \sqrt[n-1]{\left|\frac{y_0}{y}\right|} \sim e^t \quad \mbox{as} \quad t \to +\infty$.
On the other hand, note that integrals defining implicitly the curve $(x,y)$ by (\ref{timespanx}) and (\ref{timespany}) converge when $x,y \to +\infty$. Therefore there is $T>0$ depending on on the initial point $(x_0,y_0)$ such that the curve (x,y) is defined on $(-T,+\infty) \subset \mathbf R$.
It  follows that the K\"ahler metric is defined on the domain $\{z \in \mathbf C^n \mbox{ s.t. }|z|^2 > 2e^{-T}\}$ and is asymptotically Euclidean. For later use let us notice that our original coordinates $z_i$ are indeed euclidean coordinates at infinity for these metrics on the ALE end.
\bigskip

\item $x_0>0$, $\lambda = \lambda_{FYZ}$.

It is easy to check that in this case the level set $\{F(x,y) = \lambda_{FYZ}\}$ is tangent to the line $a+x+y=0$ at the point
$(n-1,-n)$. The level set  in this case gives rise to $5$ riemannian metrics. In Picture (3) below we drew the $3$ metrics corresponding to $x>0$, while we omitted the two arcs joining the origin to the admissible line at points with $y<0$. 
While these and the yellow arc in the picture do not appear to have any new special properties, we can recognize the black arc as the one corresponding to the Fu-Yau-Zhou family of metrics: indeed as for such family we know that, independently of the parameter  $a$ in Theorem \ref{FYZ}, it is asymptotically euclidean for $t\rightarrow +\infty$, while for $t\rightarrow -\infty$, 
$v_t$ is asymptotic to $-\frac{2}{n(n-1)t^2}(1-\frac{2}{n(n-1)t})$, from which it is immediate to derive that $(x(t),y(t)) \rightarrow (n-1, -n)$.
The Fu-Yau-Zhou metric has therefore to connect the same two points in the $(x,y)$-plane as the black arc, hence it must be the 
metric corresponding to it.

The metric corresponding to the blue arc does not seem to have appeared in the previous studies. Again by looking at formulae 
(\ref{timespanx}) and (\ref{timespany}), it is immediate to see that the corresponding riemannian metric is defined on the complement of a 
ball of radius $0< R_{\lambda_{FYZ}} < \infty$ and has a similar asymptotic expansion towards infinity as the Fu-Yau-Zhu metric.
It is curious to observe that it contains a unique minimal sphere (the intersection with the green line in Figure 3, as  explained in the next Section), which is unstable.

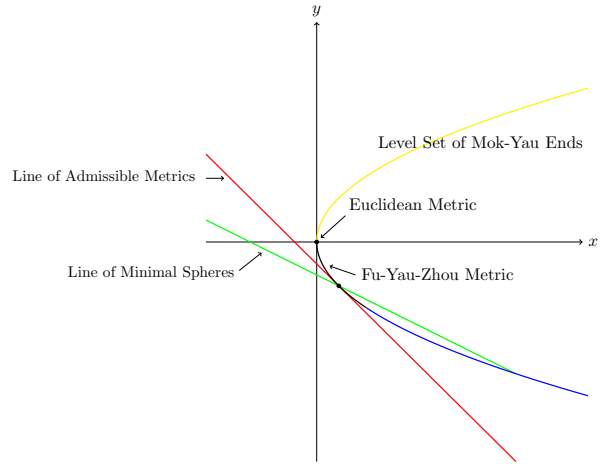
\begin{figure}
\centering

\resizebox{8cm}{!}{

\begin{tikzpicture}[scale=0.5]

\draw[->] (-5,0) -- (12,0) node[right] {$x$};
\draw[->] (0,-10) -- (0,10) node[above] {$y$};

\draw[red,thick] (-5,4) -- (9,-10);
\draw[green,thick] (-5,1) -- (9,-6);
\draw[yellow,thick,domain=0:7,smooth,variable=\x] plot ({0.25*\x*\x}, {\x});
\draw[black,thick,domain=-3:0,smooth,variable=\x] plot ({0.25*\x*\x}, {\x});
\draw[blue,thick,domain=-7:-3,smooth,variable=\x] plot ({0.25*\x*\x}, {\x});

\fill (0,0) circle (3pt);
\fill (1,-2) circle (3pt);

\node at (-14, 2.5) [above right] {\small{Line of Admissible Metrics}};
\node at (-11.5, -2) [above right] {\small{Line of Minimal Spheres}};
\node at (2.5,4) [above right] {Level Set of Mok-Yau Ends};
\node at (1.2,1.2) [above right] {Euclidean Metric};
\node at (1.75,-2) [above right] {Fu-Yau-Zhou Metric};

\draw[->][very thin] (1.3,1.2) -- (0.2,0.2);
\draw[->][very thin] (1.75,-1.5) -- (0.6,-1.1);
\draw[->][very thin] (-5,2.9) -- (-4.2,2.9);
\draw[->][very thin] (-3.5,-1.3) -- (-2.5,-0.5);
  
\end{tikzpicture}
}
\caption{Part  with $x>0$ of the level set $F=\lambda_{FYZ}$ divided into the different pieces corresponding to different metrics. } \label{fig:M3}
\end{figure}

\begin{figure}
\centering

\resizebox{8cm}{!}{

\begin{tikzpicture}[scale=0.5]

\draw (-5,0) -- (1.5,0);
\draw (2.5,0) -- (6.5,0);
\draw[->] (8.5,0) -- (12,0) node[right] {$x$};
\draw (0,-10) -- (0,1.5);
\draw (0,2.5) -- (0,4);
\draw[->] (0,6) -- (0,10) node[above] {$y$};

\draw[red,thick] (-5,4) -- (9,-10);
\draw[green,thick] (-5,1) -- (9,-6);
\draw[yellow,thick,domain=0:6.5,smooth,variable=\x] plot ({0.25*\x*\x}, {\x});
\draw[yellow,thick,domain=6.78:7,smooth,variable=\x] plot ({0.25*\x*\x}, {\x});
\draw[black,thick,domain=-3:0,smooth,variable=\x] plot ({0.25*\x*\x}, {\x});
\draw[blue,thick,domain=-6.48:-3,smooth,variable=\x] plot ({0.25*\x*\x}, {\x});
\draw[blue,thick,domain=-7:-6.78,smooth,variable=\x] plot ({0.25*\x*\x}, {\x});

\fill (0,0) circle (3pt);
\fill (1,-2) circle (3pt);

\node at (-14, 2.5) [above right] {\small{Line of Admissible Metrics}};
\node at (-11.5, -2) [above right] {\small{Line of Minimal Spheres}};
\node at (2.5,4) [above right] {Level Set of Mok-Yau Ends};
\node at (1.2,1.2) [above right] {Euclidean Metric};
\node at (1.75,-2) [above right] {Fu-Yau-Zhou Metric};

\draw[->][very thin] (1.3,1.2) -- (0.2,0.2);
\draw[->][very thin] (1.75,-1.5) -- (0.6,-1.1);
\draw[->][very thin] (-5,2.9) -- (-4.2,2.9);
\draw[->][very thin] (-3.5,-1.3) -- (-2.5,-0.5);

\draw (0,2) circle [very thin, radius=0.5];
\draw (2,0) circle [very thin, radius=0.5];
\draw (0,5) circle [very thin, radius=1];
\draw (7.5,0) circle [very thin, radius=1];
\draw (11,6.65) circle [very thin, radius=0.5];
\draw (11,-6.65) circle [very thin, radius=0.5];
\draw (8,-7) circle [very thin, radius=0.75];

\node at (0,2) {$1$};
\node at (2,0) {$1$};
\node at (7.5,0) {$2$};
\node at (11,6.7) {$3$};
\node at (11,-6.7) {$3$};
\node at (8,-7) {$4$};
\node at (0,5) {$5$};

\draw [very thin, dashed, brown] (-6,6.5) -- (10,6.5);
\draw [very thin, dashed, brown] (-6,6) -- (8.5,6);
\draw [very thin, dashed, brown] (-5.5,5.5) -- (-1,5.5);
\draw [very thin, dashed, brown] (1,5.5) -- (7,5.5);
\draw [very thin, dashed, brown] (-5.2,5) -- (-1,5);
\draw [very thin, dashed, brown] (1.2,5) -- (5.6,5);
\draw [very thin, dashed, brown] (-4.8,4.5) -- (-1,4.5);
\draw [very thin, dashed, brown] (1,4.5) -- (2.5,4.5);
\draw [very thin, dashed, brown] (-4.4,4) -- (3.7,4);
\draw [very thin, dashed, brown] (-4,3.5) -- (2.5,3.5);
\draw [very thin, dashed, brown] (-3.5,3) -- (2,3);
\draw [very thin, dashed, brown] (-3,2.5) -- (1.3,2.5);
\draw [very thin, dashed, brown] (-2.5,2) -- (-0.5,2);
\draw [very thin, dashed, brown] (0.5,2) -- (0.75,2);
\draw [very thin, dashed, brown] (-2,1.5) -- (0.4,1.5);
\draw [very thin, dashed, brown] (-1.5,1) -- (0.2,1);
\draw [very thin, dashed, brown] (-1,0.5) -- (0,0.5);
\draw [very thin, dashed, brown] (-0.75,0.2) -- (0,0.2);
\draw [very thin, dashed, brown] (-0.5,-0.2) -- (0,-0.2);
\draw [very thin, dashed, brown] (-0.3,-0.5) -- (0,-0.5);


\draw [very thin, dashed, blue] (2.25,-3) -- (2.25,-3.15);
\draw [very thin, dashed, blue] (2.5,-3.2) -- (2.5,-3.5);
\draw [very thin, dashed, blue] (2.75,-3.4) -- (2.75,-3.7);
\draw [very thin, dashed, blue] (3,-3.6) -- (3,-3.9);
\draw [very thin, dashed, blue] (3.25,-3.63) -- (3.25,-4.15);
\draw [very thin, dashed, blue] (3.5,-3.85) -- (3.5,-4.4);
\draw [very thin, dashed, blue] (3.75,-3.87) -- (3.75,-4.65);
\draw [very thin, dashed, blue] (4,-4) -- (4,-4.9);
\draw [very thin, dashed, blue] (4.25,-4.22) -- (4.25,-5.15);
\draw [very thin, dashed, blue] (4.5,-4.4) -- (4.5,-5.4);
\draw [very thin, dashed, blue] (4.75,-4.6) -- (4.75,-5.65);
\draw [very thin, dashed, blue] (5,-4.8) -- (5,-5.9);
\draw [very thin, dashed, blue] (5.25,-4.8) -- (5.25,-6.15);
\draw [very thin, dashed, blue] (5.5,-5.) -- (5.5,-6.4);
\draw [very thin, dashed, blue] (5.75,-5) -- (5.75,-6.65);
\draw [very thin, dashed, blue] (6,-5) -- (6,-6.9);
\draw [very thin, dashed, blue] (6.25,-5.2) -- (6.25,-7.15);
\draw [very thin, dashed, blue] (6.5,-5.2) -- (6.5,-7.4);
\draw [very thin, dashed, blue] (6.75,-5.2) -- (6.75,-7.65);
\draw [very thin, dashed, blue] (7,-5.4) -- (7,-7.9);
\draw [very thin, dashed, blue] (7.25,-5.5) -- (7.25,-8.22);
\draw [very thin, dashed, blue] (7.5,-5.6) -- (7.5,-6.47);
\draw [very thin, dashed, blue] (7.5,-7.7) -- (7.5,-8.47);
\draw [very thin, dashed, blue] (7.75,-5.62) -- (7.75,-6.3);
\draw [very thin, dashed, blue] (7.75,-7.8) -- (7.75,-8.75);
\draw [very thin, dashed, blue] (8,-5.74) -- (8,-6);
\draw [very thin, dashed, blue] (8,-7.8) -- (8,-9);
\draw [very thin, dashed, blue] (8.25,-5.86) -- (8.25,-6.2);
\draw [very thin, dashed, blue] (8.25,-7.8) -- (8.25,-9.15);
\draw [very thin, dashed, blue] (8.5,-5.98) -- (8.5,-6.4);
\draw [very thin, dashed, blue] (8.5,-7.62) -- (8.5,-9.40);
\draw [very thin, dashed, blue] (8.75,-5.9) -- (8.75,-9.75);
\draw [very thin, dashed, blue] (9,-6.02) -- (9,-10);

\draw [very thin, dashed, red] (9,6) -- (12,5.5);
\draw [very thin, dashed, red] (7,5.3) -- (9.5,4.85);
\draw [very thin, dashed, red] (11.5,4.45) -- (12,4.35);
\draw [very thin, dashed, red] (7.5,4.1) -- (12,3.3);
\draw [very thin, dashed, red] (3.7,3.8) -- (12,2.2);
\draw [very thin, dashed, red] (2.5,3) -- (12,1.1);
\draw [very thin, dashed, red] (1.4,2.2) -- (2.4,2);
\draw [very thin, dashed, red] (5.1,1.3) -- (7,0.95);
\draw [very thin, dashed, red] (8.2,0.78) -- (12,0);
\draw [very thin, dashed, red] (0.5,1.2) -- (6.2,0.05);
\draw [very thin, dashed, red] (8.5,-0.44) -- (12,-1.2);
\draw [very thin, dashed, red] (2.4,-0.4) -- (6.2,-1.18);
\draw [very thin, dashed, red] (8.5,-1.7) -- (12,-2.4);
\draw [very thin, dashed, red] (3.4,-1.9) -- (12,-3.7);
\draw [very thin, dashed, red] (2.4,-2.85) -- (12,-5);
\draw [very thin, dashed, red] (5,-4.45) -- (12,-6.2);

\end{tikzpicture}

}

\caption{Subdivision of the region of admissible metrics into the sub-regions described in this Section } \label{fig:M3}

\end{figure}
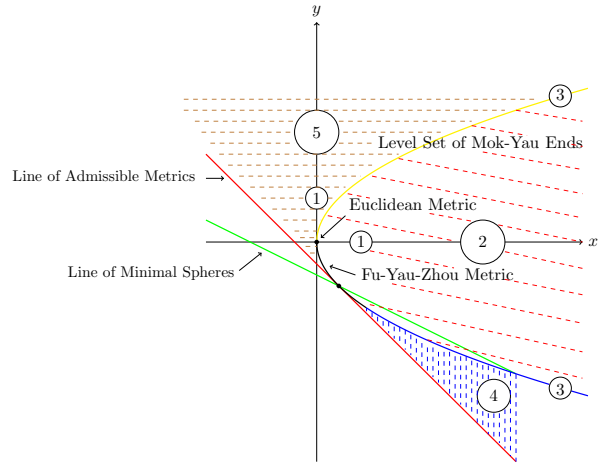

\bigskip

\item $x_0>n-1$, $\lambda \in (\lambda_{MY},+\infty)$

The curve of the vector field $\Phi_*V$ run along the graph of $\phi$. Since the graph cross the line $1+x+y=0$ at point $(w, -1-w)$ where $n-1<w<x_0$ is the solution of the equation $(1+w)^n=\lambda w^{n-1}$.
Note that integrals defining implicitly the curve $(x,y)$ by (\ref{timespanx}) and (\ref{timespany}) diverge  to $-\infty$ when $(x,y) \to (w,-1-w)$.
On the other hand, the curve $(x,y)$ converge when $x,-y \to +\infty$. Therefore there is $T_\lambda>0$ such that the curve (x,y) is defined on $(-\infty,-T_\lambda) \subset \mathbf R$.
\bigskip

\item $x<0,-1<y$ or $0<x<n-$1, $\lambda>\lambda_{MY}$

The curve of the vector field $\Phi_*V$ run along the graph of $\phi$.
Arguing as for item \eqref{item::2}, one has
\begin{equation*}
(x,y) \sim (x_0e^{-nt},y_0e^{(1-n)t}) \quad \mbox{as} \quad t \to +\infty.
\end{equation*} 
 and $ u_t = \sqrt[n]{\left|\frac{x_0}{x}\right|} = \sqrt[n-1]{\left|\frac{y_0}{y}\right|} \sim e^t \quad \mbox{as} \quad t \to +\infty$.
On the other hand, note that integrals defining implicitly the curve $(x,y)$ by (\ref{timespanx}) and (\ref{timespany}) converge when $x,y \to +\infty$. 
Since the graph cross the line $1+x+y=0$ at point $(w, -1-w)$ where $w>x_0$ if $x_0>0$, and $w<x_0$ if $x_0<0$  is the solution of the equation $1+w+\sgn(y_0)\lambda^{1/n} {w}^{1-1/n}=0$.
Note that integrals defining implicitly the curve $(x,y)$ by (\ref{timespanx}) and (\ref{timespany}) diverge  to $-\infty$ when $(x,y) \to (w,-1-w)$.

Therefore the curve (x,y) is defined on  $\mathbf R$.
It  follows that the K\"ahler metric is defined on $ \mathbf C^n\setminus\{0\}$ and is asymptotically (twice) Euclidean.
For later use let us notice that our original coordinates $z_i$ are indeed euclidean coordinates at infinity for these metrics on the ALE end.
\end{enumerate}

\section{Minimal spheres}

We now want to detect which ALE scalar flat metrics do contain a minimal $2n-1$-sphere among the loci $S^{2n-1}(r) = \partial B_n(r)=\{z \in \mathbb C^n \,|\, |z| = r\}$. In order to determinate the area of these spheres with respect of $\omega$ consider the Euler vector field
\begin{equation*}
V=\sum_{k=1}^n z_k \frac{\partial}{\partial z_k} + \bar z_k \frac{\partial}{\partial \bar z_k},
\qquad JV = \sum_{k=1}^n \sqrt{-1} z_k \frac{\partial}{\partial z_k} - \sqrt{-1}  \bar z_k \frac{\partial}{\partial \bar z_k}
\end{equation*}
so that one calculates that squared of the norm is given by
\begin{equation}\label{V}
|V|_E^2 = |z|^2=2e^t,
\qquad |V|_{\omega}^2 = 2|z|^2(f'+|z|^2f'') = 2u_{tt}.
\end{equation}
Let $\xi_E$ be the volume form induced on $\partial B_{2n}(r)$ by the Euclidean metric.
The volume form induced by the metric $\omega$ is the following
\begin{equation}\label{volform}
\xi_\omega 
= i_V \omega^n/(n!|V|_\omega) 
= 2^{n-1/2}(f')^{n-1}\sqrt{ f'+ |z|^2f''} \xi_E \\
= \frac{u_t^{n-1}\sqrt{u_{tt}}}{e^{(n-1/2)t}} \xi_E
\end{equation}  

Given that $V$ generates the flow $\theta \colon \mathbb R \times \mathbb C^n \rightarrow \mathbb C^n$ given by $\theta (\tau,  z) = e^{\tau} z$, 
it is an easy standard exercise to check that

$$V=2\partial _t = r \partial _r \, .$$

Since the Euclidean area of the sphere $S^{2n-1}(r)$ is $\frac{2\pi^nr^{2n-1}}{{n-1}!}$ one has

\begin{equation}\label{area}
\area_\omega(S^{2n-1}(r) ) 
= \frac{(2\pi)^nr^{2n-1}}{{n-1}!} \left. (f')^{n-1}\sqrt{ 2f'+ 2|z|^2f''}\right|_{|z|^2=r^2}
= \frac{(2\pi)^n}{(n-1)!} \left. u_t^{n-1}\sqrt{2u_{tt}} \right|_{t=\log(r^2/2)}.
\end{equation}

The volume of the ball $B_{2n}(r)=\{z \in \mathbb C^n \,|\, |z| \leq r\}$, with respect of $\omega$ is given by
\begin{equation*}
\vol_\omega(B_{2n}(r))=\int_{|z| \leq r} \frac{\omega^n}{n!}
=  \int_{-\infty}^{\log(r^2/2)} \area_\omega(\partial(B_{2n}(e^{t/2})) \sqrt{\frac{u_{tt}}{2}}dt
= \frac{(2\pi)^n}{n!} \left. u_t^n\right|^{t=\log(r^2/2)}_{t=-\infty}
\end{equation*}

Once we have the  area of the family of hyperspheres $S^{2n-1}(r)$, we want to understand the change in its area function under in the normal direction. By the first variation formula for the area functional we know that
\begin{equation}\label{firstvar}
    \left.\frac{d }{dt}\area_\omega(S^{2n-1}(t))\right|_{t=t_0}= - (2n-1)\int_{S^{2n-1}(t_0)} H(t_0) g(\frac{V}{2},\frac{V}{|V|_{\omega}}) \xi_\omega \,,
\end{equation}
where $H$ is the mean curvature of $S^{2n-1}(t_0)$. Since both the metric $\omega$ and the spheres are $U(n)$-invariant, we know that the mean curvature w.r.t. $\omega$ must be a function of the $t$-variable only, hence, by (\ref{volform}) and (\ref{V}), we have 
\begin{equation*}
       \left.\frac{d }{dt}\area_\omega(S^{2n-1}(t))\right|_{t=t_0}=  -(2n-1) H(t_0) \frac{\sqrt{2u_{tt}}}{2} \frac{(2\pi)^n}{(n-1)!}u_t^{n-1}\sqrt{2u_{tt}}
\end{equation*}
 
We can then prove the following 
\begin{prop}\label{thm :meancurvature}
\begin{itemize}

\item
The mean curvature of $S^{2n-1}(t)$ is given by
\begin{equation}\label{eq:of:mc:in:t}
  H(t)= \frac{-1}{(2n-1)\sqrt{2} u_{t} (u_{tt})^\frac{3}{2}}{ (2(n-1) u^2_{tt}+ u_{ttt}u_{t})}\,,  
\end{equation}

\item
A minimal $S^{2n-1}(t_0)$ is stable if and only if 
\begin{equation}\label{stable1}
\left.\left[u^2_{t}u_{tttt}-2(n-1)(4n-3)u^3_{tt}\right]\right._{t=t_{0}} \geq 0\,.
\end{equation}
\end{itemize}
\end{prop}

\begin{proof}
The first part is proved in the computation above. As for the second, differentiating twice Formula \ref{area} 
it is immediate to che that stability implies the condition above. Viceversa, inequality (\ref{stable1}) is equivalent to stability {\em within
the family of $S^{2n-1}(r)$}, yet a trivial application of the principle of symmetric criticality (\cite{Palais}) implies stability alltogether.
\end{proof}

\begin{rmk}
As a byproduct of the previous analysis we get immediately that there are no minimal hyperspheres in the family $S^{2n-1}(t)$ in $Bl_0\CC^2$ with the Burns metric and in $Bl_0\CC^2/\Gamma_{2}$
with the Eguchi-Hanson metric. Indeed, by taking the explicit expressions of the potentials in Section $3$,
 it is immediate to check that it never vanishes.

In fact, Trinca (\cite{Trinca} has recently proved that for a large class of hyperk\"ahler $4$-manifolds containing the Eguchi-Hanson space,
there are no compact minimal hypersurfaces alltogether.
\end{rmk}

In the above, we discussed the stability condition for the minimal hyperspheres $S^{2n-1}(t)$ with respect to derivative of potential function of the K\"ahler metric $\omega$. Now we translate the minimality and stability conditions in terms of $x(t))$ and $y(t)$.

\begin{prop}\label{minsp}

Let $\{(x(t),y(t))\}, t\in (a,b)$ be a solution of the system (\ref{eq::system}), hence corresponding to a scalar flat metric and lying in a level set $S_{\lambda} = \{ F(x,y) =\lambda\}$. 
Then   
\begin{itemize}
\item
$S^{2n-1}(t_0)$ is minimal w.r.t. to any metric corresponding to $\{(x(t),y(t))\}$ if and only if 
$$(n-1)x(t_0)+ny(t_0)+2n-1=0 \, .$$
\item
If $S^{2n-1}(t_0)$ is minimal w.r.t. to any metric corresponding to $\{(x(t),y(t))\}$, then it is (weakly) stable if and only if 
$$n-1 < x(t_0) \leq 2n-1 \,\, \mbox{or equivalently} \,\, -n > y(t_0) \geq 1-2n \, .$$

\end{itemize}
\end{prop}
\begin{proof}
The first part follows immediately replacing the definitions of $x$ and $y$ in terms of $u$ in Formula (\ref{eq:of:mc:in:t}).

As for the second, it is a simple lengthy computation to prove that 

\begin{align*}
  u^2_{t}u_{tttt}-2(n-1)(4n-3)u^3_{tt} = & (x+y+1)[((n-1)x+ny+2n-1)^2 - \\
  & - 6(n-1)(x+y+1)((n-1)x+ny+2n-1) + \\
 & + n(n-1)(-x-y)(x+y+1)]\,,
\end{align*}

As $t_0$ corresponds to a minimal surfaces, by the first part we know that $(n-1)x(t_0)+ny(t_0)+2n-1=0$, hence

$$(u^2_{t}u_{tttt}-2(n-1)(4n-3)u^3_{tt})(t_0) \geq 0\,\, \mbox{iff} \,\, (x+y)(t_0) \leq 0 \,\, ,$$

which implies the claim.
\end{proof}

\begin{figure}
\centering

\begin{tikzpicture}
  \draw[->] (-2.5,0) -- (5.5,0) node[right] {$x$};
  \draw[->] (0,-2.5) -- (0,5.5) node[above] {$y$};
  
  \foreach \x in {-2,...,-1,1,...,5}
    \draw (\x,-0.1) -- (\x,0.1) node[below] {$\x$};
  \foreach \y in {-5,...,-1,1,...,2}
    \draw (-0.1,\y) -- (0.1,\y) node[left] {$\y$};
  
  \draw[domain=-3.3:3.3,smooth,variable=\x,blue,thick] plot ({\x*\x/2}, {\x});
  \draw[red,thick] (3,-4) -- (-4,3);
  \draw[green,thick] (-5,1) -- (7,-5);
  \draw[brown,thick,domain=-4:4,smooth,variable=\x] plot ({0.25*\x*\x}, {\x});
  \draw[cyan,thick,domain=-4:4,smooth,variable=\x] plot ({0.3333*\x*\x}, {\x});
  \draw[yellow!50!white,thick,domain=-5:4,smooth,variable=\x] plot ({0.3*\x*\x}, {\x});
  
  \fill (1.6, -2.3) circle (2pt);
  \node at (1.6, -2.3) [below left] {$MinSt$};
    \fill (5.9, -4.4) circle (2pt);
  \node at (5.9, -4.4) [above right] {$MinUnst$};
\end{tikzpicture}

\caption{A summary of the previous analysis proving existence of level sets (e.g. the yellow one) corresponding to metrics containing stable and unstable minimal spheres} \label{fig:M4}
\end{figure}
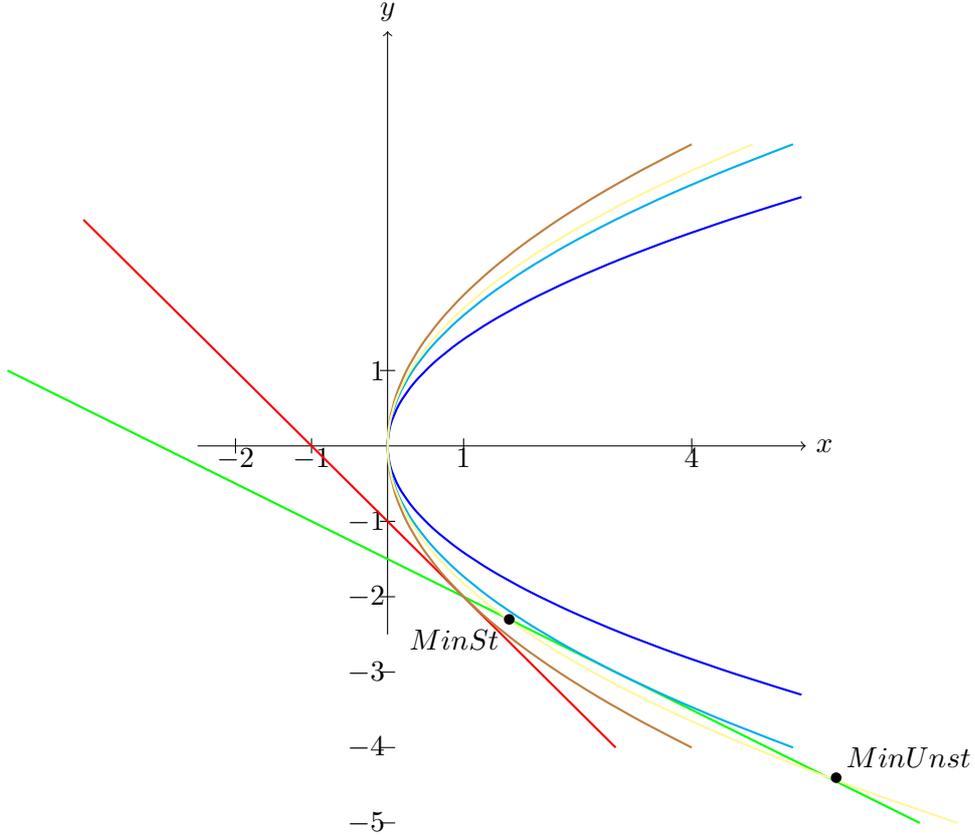

At this point we can easily observe the following

\begin{cor} \label{mainco}
\begin{enumerate}
\item
A complete $U(n)$-invariant scalar flat K\"ahler metric on a manifold without boundary does not contain any compact minimal hypersurface.
\item
There exist manifolds with boundary with complete $U(n)$-invariant scalar flat K\"ahler metrics and an asymptotically euclidean end, which contain a stable minimal hypersphere.
\end{enumerate}
\end{cor}
\begin{proof}
A straightforward application of the Maximum Principle shows that if a $U(n)$-invariant metric contains a compact minimal hypersurface, then it must contain also a minimal hypersphere $S$ in the family of $S^{2n-1}(t)$. On the other hand the above analysis easily shows that
$S$ would correspond to a point of intersection of the minimal line with the level set corresponding to such metric, but such level sets must pass through the Region ($4$) in Figure ($4$), and, as proved in Section $3$, all such sets have the property that $x(t)$ goes to infinity as $t$ goes to some finite $T_{\lambda}$, hence they do not define complete metrics without boundaries. 

As for the second point of the Corollary, it is enough to consider metrics corresponding to the arcs of level sets with $\lambda \in [2n-1,  \lambda_{FYZ})$ with $y<0$ (as the piece with $y<0$ of the yellow arc in Figure $5$). 
\end{proof}

\section{Penrose vs Hein-LeBrun Inequalities}

We are now in position to apply our analysis to the comparison between the two Penrose Inequalities available in the K\"ahler setting
mentioned in the Introduction.

Three objects evidently come into play:
\begin{itemize}
\item
the existence and volume $V$ of an outermost minimal surface $\Sigma$;
\item
the  existence and volume  of a ``canonical" divisor $\mathcal{D}$;
\item
the ADM Mass $m_{ADM}$ of an Asymptotic Locally Euclidean manifold.
\end{itemize}

As we observed in the second Section and in Proposition \ref{minsp}, in order to have a stable minimal surfaces and to have an asymptotically euclidean end,
we must look at metrics corresponding to the arcs of level sets with $\lambda \in [2n-1,  \lambda_{FYZ})$ with $y<0$. On the other hand 
as proved in Section $2$, such metrics are defined on $\CC ^n \setminus B_{R_{\lambda}}$ for some finite $R_{\lambda}$, which of course does not contain any compact divisor. The existence of a divisor can indeed happen only when and where the level set intersects the line of admissible metrics as we have seen in Subsection \ref{inco}. Hence  the presence of a divisor and of a stable minimal sphere in the family $S^{2n-1}(t)$ are mutually exclusive for $U(n)$-invariant scalar flat metrics. It is then natural to conjecture that this phenomenon is in fact more general and does not depend on the extra symmetries we imposed.

Yet, we now want to check the validity of the Penrose Inequality for the metrics we considered. To do so we first observe that the stable minimal surface we found in the previous section is indeed an outermost minimal surface: in fact, it is easy to check that the interior of any euclidean sphere is strongly convex and formula \ref{thm :meancurvature} shows that all euclidean spheres outside the stable minimal one have mean curvature of the same sign. Hence a straightforward application of the maximum principle prevents the existence of a compact minimal surface outside the stable one. We then need to compute first the volume $V(\Sigma)$ of the stable minimal sphere we found. Recall that in order to associate a metric to a (piece of) level set we fix a given point $(x_0,y_0)$ on it and we impose this to be the 
point corresponding to $t=0$, hence getting rid of undefined scaling factors.  In what follows we then choose $(x_0,y_0)$ to be the point corresponding to the unique stable minimal surface we found. 

Recall that 
\begin{equation}
V(S^{2n-1}(t) ) 
= \frac{(2\pi)^n}{(n-1)!} \left. u_t^{n-1}\sqrt{2u_{tt}} \right. \,\, .
\end{equation}

Formula (\ref{vinteg}) readily gives $u_t (0) = 1$ and $u_{tt} (0) = 1 + x_0 + y_0$, hence

\begin{equation} \label{voll}
 V(\Sigma) = \frac{(2\pi)^n}{(n-1)!} \sqrt{2(\frac{1-n+x_0}{n}}) \,\, ,
\end{equation}
where $n-1 < x_0 \leq 2n-1 $

We are then left with the problem of computing the ADM Mass for all scalar flat metrics studied above. In fact we will compute 
this quantity for a broader class of metrics, being of natural independent interest.

Recall that for any ALE scalar flat metric one has (see Lemma $7.2$ in (\cite{ap})), in a suitable choice of euclidean coordinates at infinity,
\begin{equation*}
u(t) = e^t + m\frac{e^{(2-n)t}-1}{2-n} + a +O(e^{(1-n)}), \qquad t \to +\infty,
\end{equation*}
and it is well known that $m$ is indeed the AMD mass which can then be computed by $m=\lim_{t \to +\infty} \frac{e^{(n-2)t}}{n-1}(u_t-u_{tt})$. 

It is important to notice that in all our examples, euclidean coordinates are indeed the fixed coordinates $z_i$, and in these 
coordinates metrics on the appropriate level sets are automatically ALE without need to use any scaling.

We now want to find the relationship between $m$ and the initial point $(x_0,y_0)$:

\begin{enumerate}

\item $x_0=0$

In the case studied in Example \ref{ex::BS}, for $t \to +\infty$ we have
\begin{equation*}
u_t = \sqrt[n-1]{B+e^{(n-1)t}}
= e^t+\frac{B}{(n-1)}e^{(2-n)t} + O(e^{(1-n)t}),
\end{equation*}
where $B= - y_0/(1+y_0)$..
As a consequence
\begin{align*}
u_{tt} 
&= e^{(n-1)t}(B+e^{(n-1)t})^{(2-n)/(n-1)}, \\
&= e^t(1+Be^{(1-n)t})^{(2-n)/(n-1)}, \\
&= e^t -\frac{B(2-n)}{n-1}e^{(2-n)t} + O(e^{(1-n)t}) \qquad t \to +\infty
\end{align*}
As a consequence
\begin{equation*}
m =\lim_{t \to +\infty} \frac{e^{(n-2)t}}{n-1} \left( \frac{Be^{(2-n)t}}{n-1} - \frac{B(2-n)e^{(2-n)t}}{n-1} \right) = \frac{B}{n-1}
\end{equation*}
Note that one has $m = \frac{-y_0}{((n-1)(1+y_0)}$ thus metrics related to initial point $(0,y_0)$ with positive $y_0$  has negative ADM mass.

\bigskip

\item $y_0=0$

For  example \ref{ex::HH}, for $t \to +\infty$ we have
\begin{equation*}
u_t = \sqrt[n]{B+e^{nt}}
= e^t+\frac{B}{n}e^{(1-n)t} + O(e^{-nt}),
\end{equation*}
where $B= - x_0/(1+x_0)$.
As a consequence
\begin{align*}
u_{tt} 
&= e^{nt}(B+e^{nt})^{(1-n)/n}, \\
&= e^t(1+Be^{-nt})^{(1-n)/n}, \\
&= e^t -\frac{B(1-n)}{n}e^{(1-n)t} + O(e^{-nt}) \qquad t \to +\infty
\end{align*}
As a consequence
\begin{equation*}
m =\lim_{t \to +\infty} \frac{e^{(n-1)t}}{n-1} \left( \frac{Be^{(1-n)t}}{n} - \frac{B(1-n)e^{(1-n)t}}{n} \right) = \frac{B}{n}
\end{equation*}
Note that one has $m = \frac{-x_0}{((n-1)(1+x_0)}$ thus metrics related to initial point $(0,x_0)$ with positive $x_0$  has negative ADM mass.

\bigskip

\item $x_0y_0\neq 0$

\begin{equation*}
u_t = \left|\frac{x_0}{x}\right|^{1/n}, \quad u_{tt}= \left|\frac{x_0}{x}\right|^{1/n} + x\left|\frac{x_0}{x}\right|^{1/n} + y_0|x_0|^{2/n-1}|x|^{1-2/n}
\end{equation*}
Since $x \sim x_0e^{-nt}$ for $t \to +\infty$ we have
\begin{align*}
m &= \lim_{x \to \infty} \frac{e^{(n-2)t}}{n-1}\left(-x\left|\frac{x_0}{x}\right|^{1/n} - y_0|x_0|^{2/n-1}|x|^{1-2/n}\right) \\
& = \lim_{x \to \infty} \frac{e^{(n-2)t}}{n-1}\left(-x_0e^{(1-n)t} - y_0e^{(2-n)t}\right)
\end{align*}
   whence
   \begin{equation}
   m = -\frac{y_0}{(n-1)}.
   \end{equation}

\end{enumerate}

Penrose Inequality (\ref{RiemannianPenrose}) then boils down to ask whether

$$ -\frac{y_0}{(n-1)} \geq 1+x_0+y_0 = -\frac{n+y_0}{n-1} \,\, , $$

which is trivially true in every dimension.

Corollary \ref{phl} is then proved.

\end{document}